\documentclass{amsart}
\usepackage[utf8]{inputenc}
\usepackage{amsmath}
\usepackage{amssymb}
\usepackage{caption}
\usepackage{amsthm}
\usepackage[hidelinks]{hyperref}
\usepackage{cleveref}
\crefname{lemma}{Lemma}{Lemmas}
\crefname{theorem}{Theorem}{Theorems}
\usepackage{xcolor}
\usepackage{comment}
\usepackage{graphicx}
\usepackage{tikz}
\makeatletter
\def\@settitle{\begin{center}%
  \baselineskip14\p@\relax
  \bfseries
  \uppercasenonmath\@title
  \@title
  \ifx\@subtitle\@empty\else
     \\[1ex]\uppercasenonmath\@subtitle
     \footnotesize\mdseries\@subtitle
  \fi
  \end{center}%
}
\def\subtitle#1{\gdef\@subtitle{#1}}
\def\@subtitle{}
\makeatother

\newtheorem{theorem}{Theorem}[section]

\newtheorem{lemma}[theorem]{Lemma}

\theoremstyle{remark}

\usepackage{chngcntr}
\usepackage{graphicx} 
\usepackage{float}
\counterwithout{equation}{section}
\counterwithout{theorem}{section}

\begin{document}

\title[Hyperbolic manifolds with a fixed perimeter-to-volume ratio]{Bordered hyperbolic manifolds with a fixed perimeter-to-volume ratio}

\author{Sami Douba}
\address{Mathematisches Institut der Universit\"at Bonn, Endenicher Allee 60, 53115 Bonn, Germany}
\email{douba@math.uni-bonn.de}

\author{Franco Vargas Pallete}
\address{IMPA - Instituto Nacional de Matem\'atica Pura e Aplicada, Rio de Janeiro, RJ, Brasil, 22460-320}
\email{vargas.pallete@impa.br}

\begin{abstract}
We exhibit, for any $n \geq 3$, infinitely many pairwise incommensurable complete finite-volume hyperbolic $n$-manifolds, both compact and noncompact, each with nonempty totally geodesic boundary and all sharing the same perimeter-to-volume ratio. 
\end{abstract}

\maketitle

Two Riemannian manifolds (possibly with boundary) are said to be {\em commensurable} if they share a common finite Riemannian cover. Given a complete Riemannian manifold $M$ with finite-area boundary, one defines the {\em perimeter} of $M$ to be the area of $\partial M$. For such $M$ that are moreover of finite volume, the perimeter-to-volume ratio is evidently a commensurability invariant. The goal of this note is to show that this invariant is rather crude even among hyperbolic manifolds with totally geodesic boundary.

\begin{theorem}\label{thm:main}
For every $n \geq 3$, there are infinitely many pairwise incommensurable compact (respectively, pairwise incommensurable noncompact complete finite-volume) hyperbolic $n$-manifolds with nonempty totally geodesic boundary all sharing the same perimeter-to-volume ratio.
\end{theorem}

Note that if $M$ is a compact Riemannian manifold with boundary, then $\partial M$ is compact and hence automatically of finite area. That the area of $\partial M$ remains finite in the case that  $M$ is a complete finite-volume hyperbolic manifold of dimension $\geq 3$ with totally geodesic boundary follows from a recurrence theorem of Dani~\cite{zbMATH03887957}.

Using elementary hyperbolic geometry, one easily produces uncountably many geometrically distinct complete finite-volume hyperbolic surfaces, both compact and noncompact, each with nonempty compact totally geodesic boundary and all sharing the same perimeter-to-volume ratio. On the other hand, by Mostow--Prasad rigidity~\cite{mostow1968quasi, prasad1973strong}, there are, up to isometry, only countably many complete finite-volume hyperbolic manifolds of dimension $\geq 3$ with totally geodesic boundary.

The manifolds we exhibit in the proof of Theorem~\ref{thm:main} are obtained by gluing together pieces of incommensurable arithmetic hyperbolic manifolds in the style of Gromov and Piatetski-Shapiro~\cite{MR932135}. It was observed by Raimbault~\cite{MR3090707} that the resulting hybrids are sensitive, even up to commensurability, to the pattern in which the subarithmetic building blocks are glued.

Given a complete hyperbolic manifold $M$ with totally geodesic boundary, denote by $2M$ the double of $M$ along $\partial M$. The following argument is due to Raimbault~\cite[Lem.~3.2]{MR3090707}; see also~\cite[Lem.~3.5]{MR3261631}. We sketch it for the convenience of the reader.

\begin{lemma}\label{lem:raimbault}
Let $n \geq 3$, and let $N_1$ and $N_2$ be complete finite-volume hyperbolic $n$-manifolds with totally geodesic boundary. Suppose each of the $N_j$ admits a decomposition into two submanifolds $N_j^1$ and $N_j^2$ with nonempty interior and totally geodesic boundary such that
\begin{itemize}
\item the double $2N_j^i$ is arithmetic for $i,j \in \{1,2\}$;
\item the doubles $2N_1^1$ and $2N_2^1$ are commensurable; and
\item the doubles $2N_j^1$ and $2N_j^2$ are incommensurable for $j=1,2$.
\end{itemize}
In the case $n=3$, suppose further that $\partial N_1^1$ and $\partial N_2^1$ are compact. If $N_1$ and $N_2$ are commensurable, then so are $N_1^1$ and $N_2^1$. 
\end{lemma}

\begin{proof}[Proof~of~Lemma~\ref{lem:raimbault}]
Suppose there is a common finite cover $N$ of the $N_j$ with covering maps $p_j: N \rightarrow N_j$. Then $p_j^{-1}(N_j^1) \subset p_k^{-1}(N_k^1)$ for $j \neq k$, since otherwise at least one of the components of the common submanifold $p_j^{-1}(N_j^1) \cap p_k^{-1}(N_k^2)$ of the two incommensurable arithmetic manifolds $2p_j^{-1}(N_j^1)$ and $2p_k^{-1}(N_k^2)$ would have Zariski-dense\footnote{The fundamental group of a complete finite-volume hyperbolic manifold of dimension $n\geq 4$ with (codimension-$2$) corners is Zariski-dense in $\mathrm{Isom}(\mathbb{H}^n)$. Zariski-density also holds for $n=3$ in the case of compact boundary.} fundamental group, and this is impossible since the commensurability class of an arithmetic lattice in $\mathrm{Isom}(\mathbb{H}^n)$ is already detected by any Zariski-dense subgroup. Thus, we have that $p_1^{-1}(N_1^1) = p_2^{-1}(N_2^1)$ is a common finite cover of the~$N_j^1$. 
\end{proof}

\begin{proof}[Proof~of~Theorem~\ref{thm:main}]
We first handle the compact case. Fix $n \geq 3$. Set $K= \mathbb{Q}(\sqrt{2})$, and let $\mathcal{O}_K$ be the ring of integers of $K$ (in this case, $\mathcal{O}_K = \mathbb{Z}[\sqrt{2}]$). Let $f_1$ and $f_2$ be the quadratic forms in $n+1$ variables with coefficients in $K$ given by
\begin{alignat*}{1}
x_1^2 + x_2^2 + \ldots + x_n^2 - \sqrt{2}x_{n+1}^2,& \\
17x_1^2 + x_2^2 + \ldots + x_n^2 - \sqrt{2}x_{n+1}^2,
\end{alignat*}
respectively. We choose these forms because Gelander and Levit \cite[\S4.2]{MR3261631} have shown that $f_1$ and $f_2$ are not similar over $K$. For $i=1,2$, denote by $\mathbb{H}_{f_i}$ the $f_i$-hyperboloid model of $\mathbb{H}^n$, and by $\mathrm{O}'_{f_i}(\mathbb{R}) = \mathrm{Isom}(\mathbb{H}_{f_i})$ the index-$2$ subgroup of~$\mathrm{O}_{f_i}(\mathbb{R})$ preserving~$\mathbb{H}_{f_i}$. Denote by $H_i \subset \mathbb{H}_{f_i}$ the geodesic hyperplane $\{x_1 = 0\}$. For a nonzero ideal $\mathfrak{n} \subset \mathcal{O}_K$, denote by $\Gamma_i(\mathfrak{n})$ the principal congruence subgroup of level~$\mathfrak{n}$ in $\Gamma_i := \mathrm{O}'_{f_i}(\mathcal{O}_K)$, and by $\Sigma_i(\mathfrak{n})$ the projection of $H_i$ to the hyperbolic $n$-orbifold $M_i(\mathfrak{n}):= \Gamma_i(\mathfrak{n})\backslash \mathbb{H}_{f_i}$. Note that, since we have chosen forms $f_1, f_2$ with the same coefficients for $x_2, \ldots, x_{n+1}$, the immersed compact totally geodesic hypersurfaces $\Sigma_1(\mathfrak{n})$ and $\Sigma_2(\mathfrak{n})$ of $M_1(\mathfrak{n})$ and $M_2(\mathfrak{n})$, respectively, are isometric for any choice of nonzero ideal $\mathfrak{n} \subset \mathcal{O}_K$. Now choose such an ideal~$\mathfrak{n}$ such that, for $i=1,2$, the group $\Gamma_i(\mathfrak{n})$ is torsion-free and $\Sigma_i(\mathfrak{n})$ is a properly embedded two-sided totally geodesic hypersurface in the manifold $M_i(\mathfrak{n})$. By an argument of Lubotzky~\cite{MR1390750} (see also~\cite{arXiv:2506.23994}), there is then a smaller nonzero ideal $\mathfrak{m} \subset \mathfrak{n}$ such that the union of two lifts~$\Sigma_i$ and $\Sigma_i'$ of $\Sigma_i(\mathfrak{n})$ to $M_i(\mathfrak{m})$ fails to separate~$M_i(\mathfrak{m})$. Let~$M_i$ be the manifold with nonempty totally geodesic boundary obtained by cutting~$M_i(\mathfrak{m})$ along $\Sigma_i$ and~$\Sigma_i'$. Color two of the boundary components of $M_i$ red and the other two green.

\begin{figure}
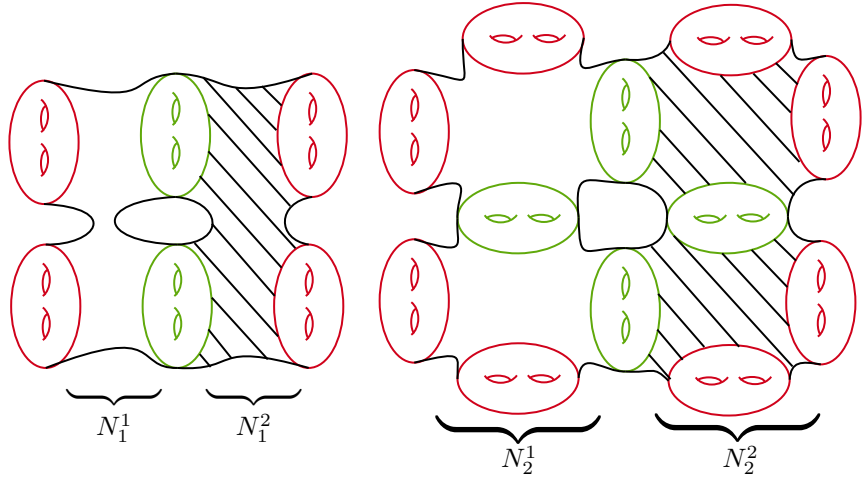

    \centering

\tikzset{every picture/.style={line width=0.75pt}} 


    \caption{Schematic representation of $N_1$ and $N_2$, respectively.}
    \label{fig:N1N2}
\end{figure}

We now define a sequence $(N_j)_{j \geq 1}$ of hyperbolic $n$-manifolds with entirely red boundary as follows. For $i=1,2$ and $j \geq 1$, let $N_j^i$ be a (connected) manifold with precisely two green boundary components obtained by doubling $M_i$ along a single green boundary component, then doubling the resulting manifold along a single green boundary component, and so on, until the resulting manifold consists of $2^{j-1}$ copies of $M_i$. Then the double $2N_j^i$ is commensurable to $M_i(\mathfrak{m})$. In particular, we have that $2N_j^1$ and $2N_j^2$ are incommensurable arithmetic manifolds since the forms $f_1$ and $f_2$ are not similar over $K$; see \cite[\S2.6]{MR932135}. Now set $N_j$ to be any manifold with entirely red boundary obtained by gluing the $N_j^i$ along the green boundary components via isometries (see Figure~\ref{fig:N1N2}). Then $\mathrm{area}(\partial N_j) = 2^{j+1}\mathrm{area}(\Sigma_1)$ and $\mathrm{vol}(N_j)=2^{j-1}(\mathrm{vol}(M_1)+\mathrm{vol}(M_2))$, so that the $N_j$ all share the same perimeter-to-volume ratio~$\frac{4 \cdot \mathrm{area}(\Sigma_1)}{\mathrm{vol}(M_1)+\mathrm{vol}(M_2)}$. Moreover, as the perimeter-to-volume ratio of $N^1_j$ is $(1+2^{-j})\frac{\mathrm{area}(\Sigma_1)}{\mathrm{vol}(M_1)}$, we have that for $k > j \geq 1$ the perimeter-to-volume ratio of`$N_k^1$ is strictly smaller than that of $N_j^1$, and hence $N_k$ is not commensurable to $N_j$ by Lemma~\ref{lem:raimbault}.

Note that, by the theorem of Borel and Harish-Chandra \cite{MR147566, MR141672}, the manifolds~$M_i$, and hence also the manifolds~$N_j$, obtained in the above manner are compact since the chosen forms $f_i$ were $K$-anisotropic. We next explain how to adjust the above argument to obtain cusped examples in the case $n \geq 4$. In this case, we take $K= \mathbb{Q}$ (so that $\mathcal{O}_K = \mathbb{Z}$), and instead use the forms $f_1$ and $f_2$ given by
\begin{alignat*}{1}
x_1^2 + x_2^2 + \ldots + x_n^2 - 2x_{n+1}^2,& \\
5x_1^2 + x_2^2 + \ldots + x_n^2 - 2x_{n+1}^2,
\end{alignat*}
respectively; Gelander and Levit \cite{MR3261631} have again shown that $f_1$ and $f_2$ are not similar over $K=\mathbb{Q}$. We now proceed precisely as in the compact case, but the manifolds~$M_i$ one obtains in the process, and hence also the manifolds $N_j$, have cusps since the~$f_i$ are $\mathbb{Q}$-isotropic. 

Finally, we explain how to obtain cusped examples in the case $n=3$. We again take $K=\mathbb{Q}$, but we instead take $f_1$ and $f_2$ to be the forms
\begin{alignat*}{1}
x_1^2 + x_2^2 + x_3^2 - 7x_4^2,& \\
7x_1^2 + x_2^2 + x_3^2 - 7x_4^2,
\end{alignat*}
respectively. That $f_1$ and $f_2$ are not similar over $K=\mathbb{Q}$ can be seen for instance from the fact that $f_1$ is $\mathbb{Q}$-anisotropic\footnote{Since $-7$ is a square in the $2$-adics $\mathbb{Q}_2$ by Hensel's lemma, and since the stufe of $\mathbb{Q}_2$ is $4$, the form $f_1$ is indeed $\mathbb{Q}_2$-anisotropic.} while $f_2$ is evidently not. In particular, while the resulting manifold $M_1$ will be compact in this case, so that the surfaces $\Sigma_i$ and $\Sigma_i'$ will also be compact, the manifold $M_2$ will be cusped. We may now proceed precisely as in the compact case, since Lemma~\ref{lem:raimbault} still applies to distinguish the commensurability classes of the resulting manifolds $N_j$.
\end{proof}

We conclude with some remarks. Observe that a pair of complete finite-volume manifolds with finite perimeters have the same isoperimetric ratio if and only if the ratio of their volumes agrees with the ratio of their perimeters, while commensurability moreover implies that this common ratio is rational. It is clear that the latter also holds for each family $\{N_j\}_j$ of incommensurable manifolds constructed above.

If $M$ is a complete finite-volume hyperbolic $n$-manifold with nonempty totally geodesic boundary, then
the ideal boundary of the universal cover $\widetilde{M}$ on the conformal sphere $\mathbb{S}^{n-1}$ at infinity, where $\widetilde{M}$ is viewed as a closed convex subset of~$\mathbb{H}^n$, is a {\em Schottky set}, that is, the complement in $\mathbb{S}^{n-1}$ of a union of disjoint round open balls. Two such hyperbolic manifolds are then commensurable precisely when their associated Schottky sets differ by a conformal map of $\mathbb{S}^{n-1}$. It follows from a rigidity theorem of Bonk--Kleiner--Merenkov~\cite{zbMATH05549754} that the sequences of hyperbolic manifolds sharing the same perimeter-to-volume ratio that are constructed in the proof of Theorem~\ref{thm:main} indeed have the property that their associated Schottky sets pairwise fail to be quasisymmetric (in the compact case, this already follows from earlier work of Frigerio~\cite{MR2239451}, as well as from unpublished work of Kapovich--Kleiner--Leeb--Schwartz; see the second footnote in~\cite{zbMATH01571175}). 

Besides the perimeter-to-volume ratio, another commensurability invariant of a complete finite-volume hyperbolic manifold with nonempty totally geodesic boundary is of course the Hausdorff dimension of the associated Schottky set. We however remain unaware of a single pair of incommensurable such manifolds of dimension $\geq 3$, let alone an infinite family of pairwise incommensurable such manifolds, that demonstrably fail to be distinguished by the latter invariant. We are also not aware of a complete hyperbolic manifold $M$ of dimension $\geq 3$ sharing the same perimeter-to-volume ratio with an incommensurable such manifold of the same dimension but possessing no properly immersed totally geodesic hypersurface contained entirely  in the interior of $M$. The absence of such a hypersurface within $M$ is equivalent to the absence of nonperipheral round hyperspheres within the Schottky set associated to $M$, as was recently established in full generality by Lee--Oh~\cite[Cor.~1.2]{arXiv:2511.09377} using Ratner's measure classification theorem~\cite{MR1135878} and work of Dani--Margulis~\cite{zbMATH00480583}; see also the previous works~\cite[Appendix~B]{zbMATH06786947}, \cite[\S10]{zbMATH07457815}.

It follows essentially from strong approximation that, up to diminishing the ideal~$\mathfrak{m}$ in the proof of Theorem~\ref{thm:main}, the common perimeter-to-volume ratio in Theorem~\ref{thm:main} can be made arbitrarily small; see Belolipetsky--Weinberger~\cite[Thm.~4.1]{MR4768580} (in particular the last page of the proof). One can then apply an inequality of Buser~\cite[Thm.~7.1]{zbMATH03789440} and identities of Elstrodt--Patterson--Sullivan~\cite[Thm.~2.17]{zbMATH03996555} and Sullivan~\cite{zbMATH03685853, zbMATH03903608} to see that the Hausdorff dimensions of the associated Schottky sets can be made arbitrarily close to the dimension of the ambient sphere.

On the other hand, Miyamoto \cite{MR1293303} showed that for $n=3,4$, there is a unique commensurability class of complete finite-volume hyperbolic $n$-manifolds with totally geodesic boundary maximizing the perimeter-to-volume ratio among all such manifolds. For $n=3$, the maximizers are precisely those whose associated Schottky set in $\mathbb{S}^2$ is the Apollonian gasket;  accordingly, we will call any such hyperbolic $3$-manifold with totally geodesic boundary {\em Apollonian}.

 J. Raimbault has informed us that it follows from his work~\cite{zbMATH07665059} with B. Petri that one should expect the perimeter-to-volume ratio of a generic compact hyperbolic $3$-manifold with nonempty totally geodesic boundary to be large, in the following sense. In \cite{zbMATH07665059}, Petri and Raimbault introduced a model for random compact $3$-manifolds with nonempty boundary; a manifold in their model is obtained by randomly gluing truncated tetrahedra along their hexagonal facets. Among other features, they show that, with probability tending to one as the number of tetrahedra tends to infinity, such a manifold is connected and admits a hyperbolic metric with totally geodesic boundary. Moreover, it follows from their hyperbolization argument that, given any $\epsilon >0$, the perimeter-to-volume ratio of such a manifold is asymptotically almost surely within $\epsilon$ from that of an Apollonian manifold as the number of tetrahedra tends to infinity. Indeed, Petri and Raimbault show, in a sense that they make precise, that as the number of tetrahedra tends to infinity, a random $3$-manifold in their model converges to the universal cover of an Apollonian manifold, that is, to the convex hull in $\mathbb{H}^3$ of the Apollonian gasket.

\subsection*{Acknowledgements} We are grateful to Fernando Al Assal, Mikhail Belolipetsky, Jean Raimbault, and Matthew Stover for helpful discussions. This work is based on discussions conducted while SD was in residence at the Simons Laufer Mathematical Sciences Institute in Berkeley, California (National Science Foundation No. DMS2424139); the authors are grateful to the organizers of the twin programs held there in the Spring 2026 semester, and to the institute for its hospitality and support.

\bibliography{isoperimetricbib}{} 
\bibliographystyle{siam}

\end{document}